\documentclass[12pt]{amsart}
\usepackage{latexsym}
\usepackage[all]{xy}
\usepackage{amsmath,amssymb,amsfonts}
\usepackage[colorlinks]{hyperref}
\usepackage{tikz-cd}
\usepackage[utf8]{inputenc}
\UseRawInputEncoding
\def\<{{\langle}}
\def\>{{\rangle}}

\def\note#1{{}}

\def\note#1{}

\def\beq{\begin{equation}}
\def\eeq{\end{equation}}

\newtheorem{theorem}{Theorem}[section]

\newtheorem{proposition}[theorem]{Proposition}
\newtheorem{corollary}[theorem]{Corollary}

\theoremstyle{definition}
\newtheorem{definition}[theorem]{Definition}

\newtheorem{remark}[theorem]{Remark}
\newtheorem{example}[theorem]{Example}

\theoremstyle{Definition and Notation}

\begin{document}


\title[On Gelfand graded commutative rings ]{On Gelfand graded commutative rings }

\author[M. Aqalmoun]{Mohamed Aqalmoun}
\address[Mohamed Aqalmoun]{Department of Mathematics, Higher Normal School, Sidi Mohamed Ben Abdellah University, Fez, Morocco.}
\email{ maqalmoun@yahoo.fr}



\begin{abstract}
This paper deals with the graded commutative rings in which every homogeneous prime ideal is contained in a unique homogeneous maximal ideal called Gelfand graded ring. The purpose is to establish some topological and algebraic characterizations of these rings, one of which is the algebraic analogue of the Urysohn's lemma. Finally we look at a special class of those graded rings called pm$^+$  graded rings which can be viewed as graded ring with a Gelfand strong property.
 \end{abstract}
\maketitle
\section{Introduction }
Let $G$ be a group with identity $e$ and $R$ be a commutative ring with unit. Then $R$ is called a $G$-graded ring if  there exist additive subgroups $R_g$ of $R$ indexed by elements $g\in G$ such that $R=\bigoplus_{g\in G}R_g$ and $R_gR_{g'}\subseteq R_{gg'}$ for all $g,g'\in G$, where $R_gR_{g'}$ consists of the finite sums of ring products $ab$ with $a\in R_g$ and $b\in R_{g'}$. The elements of $R_g$ are called homogeneous elements of $R$ of degree $g$. The homogeneous elements of the ring $R$ are denoted by $h(R)$, i.e. $h(R)=\cup_{g\in G}R_g$. If $a\in R$, then the element $a$ can be written uniquely as $\sum_{g\in G}a_g$, where $a_g\in R_g$ is called the $g$-component of $a$ in $R_g$.\par 
Let $R$ be a $G$-graded commutative ring and $I$ be an ideal of $R$. Then $I$ is called a graded ideal of $R$ if $I=\bigoplus_{g\in G}(I\cap R_g)$. If $I$ is a $G$-graded ideal of $R$, then the quotient ring $R/I$ is a $G$-graded ring. Indeed $R/I=\bigoplus_{g\in G}(R/I)_g$ where $(R/I)_g=(R_g+I)/I=\{x+I\ | x\in R_g\}$. Let $S\subseteq h(R)$ be a multiplicatively closed subset of $R$. Then the ring of fractions $S^{-1}R$ is a $G$-graded ring. Indeed $S^{-1}R=\bigoplus_{g\in G}(S^{-1}R)_g$ where $(S^{-1}R)_g=\{r/s\ \ | \ r\in h(R), s\in S \text{ and } \deg r=g\deg s \}$.\par 
A $G$-graded ideal $P$ of a $G$-graded ring $R$ is called $G$-graded prime ideal or homogeneous prime ideal of $R$ if $P\ne R$ and if whenever $r$ and $s$ are homogeneous elements of $R $ such that $rs\in P$, then either $r\in P$ or $s\in P$. The $G$-graded prime spectrum or homogeneous prime spectrum of $R$ is the set of all $G$-graded prime ideals of $R$, it is denoted by $G\mathrm{Spec}(R)$. A $G$-graded ideal $M$ of $R$ is said to be $G$-graded maximal ideal of $R$ if $M\ne R$ and if $J$ is a $G$-graded ideal of $R$ such that $M\subseteq J\subseteq R$, then $J=M$ or $J=R$, the set of all $G$-graded maximal ideals of $R$ is denoted by $G\mathrm{Max}(R)$. Note that, by Zorn's Lemma each proper $G$-graded ideal of $R$ is contained in a $G$-graded maximal ideal of $R$.    \par 
For a commutative ring $R$ and ideal $I$  of $R$, the variety of $I$ is the subset $V(I):=\{P\in \mathrm{Spec}R\ | I\subseteq P\}$. Then the collection $\{V(I)\ | I \text{ ideal of } R\}$ satisfies the axioms for the closed sets of a topology on $\mathrm{Spec}R$, called the Zariski topology. Note that for each ideal $I$, the set $\mathrm{Spec}R\setminus V(I)$ is an open subset, it is denoted $D(I)$. Recall that, the collection $D(f)$ where $f$ runs through $R$ forms a basis of opens for the Zariski topology. Similarly, let $R$ be a $G$-graded commutative ring and $I$ be a $G$-graded ideal of $R$, the $G$-variety of $I$ is the subset $V_G(I):=\{P\in G\mathrm{Spec}R\ | I\subseteq P\}$. The collection $\{V_G(I)\ | I \text{ is a } G\text{-graded ideal of } R\}$ satisfies the axioms for the closed sets of a topology on $G\mathrm{Spec}(R)$, called the Zariski topology, for more details see  \cite{Refai}, \cite{Daran}. Note that, the collection $D(r)$ where $r$ runs through $h(R)$ forms a basis of opens for the Zariski topology.\par 
A commutative ring $R$ is said to be a Gelfand ring (or, pm-ring) if each prime ideal of $R$ is contained in a unique maximal ideal of $R$. This class of rings has been originally introduced and studied by De Marco and Orsatti \cite{Marco}. Recall that a subspace $Y$ of a topological space $X$ is called a retract of $X$ if there exists a continuous map $\varphi :X\to Y$ such that for all $y\in Y$, $\varphi(y)=y$, and such a map $\varphi$ is called a retraction. One of the topological characterization of the Gelfand ring is that; a ring $R$ is a Gelfand if and only if $\mathrm{Max}(R)$ is a Zariski retract of $\mathrm{Spec}(R)$ see \cite{Marco}. Other topological and algebraic characterizations can be found in \cite{Marco},\cite{Maria}, \cite{Tariz}, and can be summarized as follow: 
\begin{theorem}
For a commutative ring $R$ the following statements are equivalent.
\begin{enumerate}
\item $R$ is a Gelfand ring.
\item $\mathrm{Max}(R)$ is a Zariski retract of $\mathrm{Spec}(R)$.  
\item $\mathrm{Spec}(R)$ is a normal space with respect to the Zariski topology.
\item For each $M\in \mathrm{Max}(R)$, $\mathrm{Spec}(R_M)$ is a closed subset of $\mathrm{Spec}(R)$.
\item If $M\neq M'\in \mathrm{Max}(R)$ then there exists $a\in R\setminus M$ and $b\in R\setminus M'$ such that $ab=0$.
\item If $a,a'\in R$ with $a+a'=1$, then there exists $b,b'\in R$ such that $(1-ab)(1-a'b')=0$.
\end{enumerate} 
\end{theorem}
In this paper, we investigate a generalization of Gelfand commutative ring to the class of $G$-graded commutative rings that we call Gelfand $G$-graded commutative ring. Let $R$ be a $G$-graded commutative ring, then $R$ is called a Gelfand $G$-graded ring if every homogeneous prime ideal of $R$ is contained in a unique maximal graded ideal of $R$. Note that if $G=\{e\}$ is the trivial group, then Gelfand $G$-graded ring and Gelfand ring coincide. Our aims is to establish a topological and algebraic characterizations of a such class of $G$-graded commutative rings. It is shown that a $G$-graded commutative ring $R$ is Gelfand $G$-graded ring if and only if $G\mathrm{Max}(R)$ is a Zariski retract of $G\mathrm{Spec}(R)$ if and only if $G\mathrm{Spec}(R)$ is a normal space. Also, we consider algebraic characterization of $G$-graded commutative ring, which yields a surprising characterization as follow, $R$ is a Gelfand $G$-graded ring if and only if $R_e$ is a Gelfand ring. Finally, we look at a special class of these rings called pm$^+$ graded rings which are $G$-graded rings with a  Gelfand strong property. A $G$-graded ring is called pm$^+$ $G$-graded ring if for each homogeneous prime ideal $P$ of $R$ the homogeneous prime ideals of $R$ containing $P$ form a chain. We draw some connection between Gelfand graded ring and pm$^+$ graded ring.      
\section{Gelfand graded rings}
We start this section by the following result which will be relevant in the sequel.
\begin{proposition}\label{ExistsPrime}
 Let $R$ be a $G$-graded commutative ring and $I$ be a $G$-graded ideal of $R$. Let $(I_{\alpha})_{\alpha\in \Lambda}$ be a collection of proper $G$-graded ideals of $R$ such that $I\subseteq \cup_{\alpha\in \Lambda}I_{\alpha}$ with the property that, if $ab\in I$ where $a,b\in h(R)$ then $a\in \cup_{\alpha\in \Lambda}I_\alpha$ or $b\in \cup_{\alpha\in \Lambda}I_\alpha$. Then there exists a homogeneous prime ideal $Q$ of $R$ such that $I\subseteq Q$ and $Q\cap h(R)\subseteq \cup_{\alpha\in \Lambda}I_\alpha$. 
\end{proposition} 
\begin{proof}
Note that by hypothesis if $r_1,\ldots,r_n\in h(R)$ such that $r_1\ldots r_n\in I$, then $r_i\in \cup_{\alpha\in \Lambda}I_\alpha$ for some $i$. Now, consider 
$$S=\{r_1\ldots r_n\ / \  n\ge 1, \ r_i\in h(R)-\cup_{\alpha\in \Lambda}I_\alpha\} $$ 
Observe that $S$ is a multiplicatively closed subset of $R$ and $S\subseteq h(R)$. Moreover $I\cap S=\emptyset$ . By Zorn's Lemma (see \cite[Lemma $4.7$]{Wu}), there exists a homogeneous prime ideal $Q$ of $R$ such that $I\subseteq Q$ and $Q\cap S=\emptyset$. If $r\in Q\cap h(R)$, then $r\not\in S$, that is $r\in \cup_{\alpha\in \Lambda}I_\alpha$. Thus $Q\cap h(R)\subseteq \cup_{\alpha\in \Lambda}I_{\alpha}$. 
\end{proof}
\begin{remark}
Let $R$ be a $G$-graded commutative ring and $I\subseteq J$ be a $G$-graded ideals of $R$ such that if $ab\in I $ where $a,b\in h(R)$ then $a\in J$ or $b\in J$. Then there exists a homogeneous prime ideal $Q$ of $R$ such that $I\subseteq Q$ and $Q\cap h(R)\subseteq J$ that is $I\subseteq Q\subseteq J$.
\end{remark}
The following result is the avoidance property for the collection of graded maximal ideals containing a fixed graded ideal.
\begin{proposition}\label{Avoid}
Let $R$ be a $G$-graded commutative ring and $I$ be a $G$-graded ideal of $R$ and $Q\in G\mathrm{Spec}R$ such that $Q\cap h(R)\subseteq \cup_{M\in \Lambda_I}M$ where $\Lambda_I$ is the collection of all $G$-graded maximal ideals of $R$ containing $I$. Then $Q\subseteq M$ for some $M\in \Lambda_I$.
\end{proposition}
\begin{proof}
Take $J=Q+I$, then $J$ is a $G$-graded ideal of $R$. Moreover $J\ne R$. In fact; if $1\in J$, then $1=a+b$ where $a\in Q$ and $b\in I$, thus $1=a_e+b_e$ where $a_e$ (respectively $b_e$) is the $e$-component of $a$ (respectively of $b$). Since $I$ and $Q$ are $G$-graded ideals of $R$, we have  $a_e\in Q$ and $b_e\in I$. It follows that $a_e\in M'$ for some $M'\in \Lambda_I$, thus $1=a_e+b_e\in M'$, a contradiction. As $J$ is proper $G$-graded ideal of $R$, it is contained in a $G$-graded maximal ideal $M$. Now, observe that $I\subseteq M$ and $Q\subseteq M$.  
\end{proof}
\begin{definition}
Let $R$ be a $G$-graded commutative ring. We say that $R$ is a Gelfand graded ring if each homogeneous prime ideal of $R$ is contained in a unique graded maximal ideal of $R$. 
\end{definition}
\begin{example}
\begin{enumerate}
\item A local graded commutative ring is clearly a Gelfand graded ring.
\item Let $R$ be a Gelfand commutative ring and $S=\{\begin{pmatrix}
a&b\\ 0&a
\end{pmatrix} \ | a,b\in R\}$, then $S$ is a $\Bbb Z_2$-graded ring with $S_0=\{\begin{pmatrix}
a&0\\ 0&a
\end{pmatrix} \ | a\in R\}$ and $S_1=\{\begin{pmatrix}
0&b\\ 0&0
\end{pmatrix} \ | b\in R\}$. Then $S$ is a Gelfand $\Bbb Z_2$-graded ring since $\Bbb Z_2\mathrm{Spec}(S)$ and $\mathrm{Spec}(R)$ are homeomorphic with respect to the Zariski topologies, see \cite{Aqa}.
\end{enumerate}
\end{example}
The following Theorem is a topological characterization of Gelfand graded ring it terms of retraction. 
\begin{theorem}
Let $R$ be a $G$-graded commutative ring. The following statements are equivalent.
\begin{enumerate}
\item $R$ is a Gelfand $G$-graded ring.
\item $G\mathrm{Max}(R)$ is a Zariski retract of $G\mathrm{Spec}(R)$.
\end{enumerate} 
\end{theorem}
\begin{proof}
$(1)\Rightarrow (2)$. Let $\varphi:G\mathrm{Spec}(R)\to G\mathrm{Max}(R)$ where $\varphi(P)$ is the unique $G$-graded maximal ideal containing $P$. Let $F$ be a closed subset of $G\mathrm{Max}(R)$. It is easy to see that $F$ is the set of all $G$-graded maximal ideals of $R$ containing $J:=\cap_{M\in F}M$. We must show that $\varphi^{-1}(F)$ is a closed subset of $G\mathrm{Spec}R$. It is enough to show that $\varphi^{-1}(F)=V_G(I)$ where $I=\cap_{\varphi(P)\in F} P$.\par 
Let $P\in \varphi^{-1}(F)$, then $\varphi(P)\in F$, thus $I\subseteq P$ that is $ P\in V_G(I)$.\par 
Let $Q\in V_G(I)$ that is $I\subseteq Q$. In particular $I\subseteq Q\cap (\cup_{M\in F}M)$. Now, let $x,y\in h(R) $ with $xy\in I$ and $y\not\in Q\cap (\cup_{M\in F}M)$. We must show that $x\in Q\cap (\cup_{M\in F}M)$.\par 
\textbf{First case; }$y\not\in Q$. In this case $y\not\in I$ since $I\subseteq Q$, and $x\in Q$ since $Q$ is homogeneous prime. But $y\not\in I=\cap_{\varphi(P)\in F}P$ implies that $y\not\in P_0$ for some $P_0\in G\mathrm{Spec}R$ with $\varphi(P_0)\in F$. Since $xy\in I\subseteq P_0$ and $y\not\in P_0$, we get $x\in P_0\subseteq \varphi(P_0)\in F$. Thus $x\in\cup_{M\in F}M$. Therefore $x\in Q\cap (\cup_{M\in F}M)$.\par 
\textbf{Second case;} $y\not\in \cup_{M\in F}M$, that is $y\not\in M$ for all $M\in F$.  If $P\in \varphi^{-1}(F)$, then $P\subseteq \varphi(P)\in F$ so $y\not\in P$. We see that $xy\in I$, hence $xy\in P$ for all $P\in \varphi^{-1}(F)$, so $x\in P$ for all $P\in \varphi^{-1}(F)$. Therefore $x\in I$, thus $x\in Q\cap (\cup_{M\in F}M)$.\\ 
 Now, $I\subseteq \cup_{M\in F}(Q\cap M)$, by Lemma \ref{ExistsPrime}, there exists homogeneous prime ideal $Q'$ of $R$ such that $I\subseteq Q'$ and $Q'\cap h(R)\subseteq \cup_{M\in F}(Q\cap M)$.   On one hand $Q'\cap h(R)\subseteq Q$, implies that $Q'\subseteq Q$. On other hand $ Q'\cap h(R)\subseteq \cup_{M\in F}M$ and $F$ is the set of all graded maximal ideals of $R$ containing $J$, by Lemma \ref{Avoid}, $Q'\subseteq M$ for some $M\in F$. It follows that $\varphi(Q')=M$, therefore $\varphi(Q)=\varphi(Q')\in F$, that is $Q\in \varphi^{-1}(F)$.\par 
$(2)\Rightarrow (1)$. Assume that $G\mathrm{Max}(R)$ is a retract of $G\mathrm{Spec}(R)$. Let $\phi :G\mathrm{Spec}(R)\to G\mathrm{Max}(R)$ be any retraction.  Let $P\in G\mathrm{Spec}(R)$ and $M\in G\mathrm{Max}(R)$ containing $P$. Since $\phi(P)$ is a $G$-graded maximal ideal of $R$, $V_G(\phi(P))=\{\phi(P)\}$ is a closed subset of $G\mathrm{Max}(R)$. Thus $\phi^{-1}(\{\phi(P)\})$ is a closed subset of $G\mathrm{Spec}(R)$. It follows that $M\in \phi^{-1}(\{\phi(P)\})$ since $P\in \phi^{-1}(\{\phi(P)\})$ and closed subset are closed under specialization. Therefore $M=\phi(M)=\phi(P)$, hence $\phi(P)$ is the unique $G$-graded maximal ideal of $R$ containing $P$. 
\end{proof}
The following is  a characterization of a Gelfand graded ring by the closure of $G\mathrm{Spec}(R_M)$ in $G\mathrm{Spec}(R)$ where $M\in G\mathrm{Max}(R)$.  
\begin{theorem}
Let $R$ be a $G$-graded commutative ring. The following statements are equivalent.
\begin{enumerate}
\item $R$ is a Gelfand $G$-graded ring.
\item For every $G$-graded maximal ideal $M$ of $R$, $\{P\in G\mathrm{Spec}(R)\ | \ P\subseteq M\}$ is a Zariski closed subset of $G\mathrm{Spec}(R)$.
\end{enumerate}
\end{theorem}
\begin{proof}
$(1)\Rightarrow (2)$. Since $R$ is a Gelfand $G$-graded ring, the map $\varphi:G\mathrm{Spec}(R)\to G\mathrm{Max}(R)$, where $\varphi(P)$ is the unique $G$-graded maximal ideal of $R$ containing $P$, is a Zariski retract. Now, observe that $\{P\in G\mathrm{Spec}(R)\ | \ P\subseteq M\}=\varphi^{-1}(\{M\})$ which is a closed subset of $G\mathrm{Spec}(R)$ since $\varphi$ is continuous and $\{M\}$ is a closed subset of $G\mathrm{Max}(R)$. \par 
$(2)\Rightarrow (1)$. For a $G$-graded maximal ideal $M$ of $R$, we set $Y_M:=\{P\in G\mathrm{Spec}(R)\ | P\subseteq M\}$. Let $P\in G\mathrm{Spec}(R)$ and $M$ be a $G$-graded maximal ideal of $R$ containing $P$. Since $Y_M$ is a closed subset of $G\mathrm{Spec}(R)$, so is $V_G(P)\cap Y_M=\{Q\in G\mathrm{Spec}(R)\ | \ P\subseteq Q\subseteq M\}$. Therefore $V_G(P)\cap Y_M=V_G(I)$ where $I=\cap_{Q\in V_G(P)\cap Y_M}=P$, that is $ V_G(P)\cap Y_M=V_G(P)$. Now, if $M'$ is a $G$-graded maximal ideal of $R$ containing $P$, then $M'\in V_G(P)=V_G(P)\cap Y_M$, so that $M'\subseteq M$, thus $M=M'$.  
\end{proof}
Next, we investigate the separation criterion of Gelfand graded ring. Before, we start with its  algebraic aspect.  
\begin{theorem}\label{Disjoint}
Let $R$ be $G$-graded commutative ring. The following statements are equivalent.
\begin{enumerate}
\item $R$ is a Gelfand $G$-graded ring.
\item If $F$ and $F'$ are disjoint closed subsets of $G\mathrm{Max}(R)$, then there exists $r,r'\in h(R)$ with $rr'=0$, such that $r\not\in \cup_{M\in F} M$ and $r'\not\in \cup_{M'\in F'}M'$.
\end{enumerate}
\end{theorem}
\begin{proof}
$(1)\Rightarrow (2)$.  Let $F,F'$ be a disjoint closed subsets of $G\mathrm{Max}(R)$. Then there is no homogeneous prime ideal contained in $(\cup_{M\in F}M)\cap (\cup_{M'\in F'}M')$. In fact, if $Q$ is a homogeneous prime ideal of $R$ with $Q\subseteq (\cup_{M\in F}M)\cap (\cup_{M'\in F'}M')$, then $Q\subseteq \cup_{M\in F}M$ and $Q\subseteq \cup_{M'\in F'}M'$. By \ref{Avoid}, $Q\subseteq M$ and $Q\subseteq M'$ for some $M\in F$ and $M'\in F'$, which is not possible since $M\ne M'$ and $R$ is a Gelfand $G$-graded ring. Now, observe that $0\subseteq (\cup_{M\in F}M)\cap (\cup_{M'\in F'}M')=\cup_{M\in F,M'\in F'}M\cap M'$, by Lemma \ref{ExistsPrime}, there exists homogeneous elements $r,r'\in R$ with $rr'=0$ and $r,r'\displaystyle\not\in (\cup_{M\in F}M)\cap (\cup_{M'\in F'}M')$. We see that $r\not\in \cup_{M\in F}M$ or $r\not\in \cup_{M'\in F'}M'$. Without lost of generality, we may assume that $r\not\in \cup_{M\in F}M$. So $r'\in M$ for all $M\in F'$ since $rr'=0\in M$. It follows that $r'\in \cup_{M\in F}M$, therefore $r'\not\in \cup_{M'\in F'}M'$.\par 
$(2)\Rightarrow (1)$. Let $P\in G\mathrm{Spec}(R)$. Assume $P\subseteq M\cap M'$ where $M\ne M'$ are $G$-graded maximal ideals of $R$. Observe that $\{M\}$ and $\{M'\}$ are disjoint closed subsets of $G\mathrm{Max}(R)$. By hypothesis there exists homogeneous elements $r,r'\in R$ with $rr'=0$, $r\not\in M$ and $r'\not\in M'$. But $rr'=0\in P$ implies that $r\in P$ or $r'\in P$, so $r\in M\cap M'$ or $r'\in M\cap M$ a contradiction.
\end{proof}
\begin{corollary}\label{Max}
Let $R$ be $G$-graded commutative ring. The following statements are equivalent.
 \begin{enumerate}
 \item $R$ is a $G$-graded Gelfand ring.
 \item If $M\ne M'\in G\mathrm{Max}(R)$, then $ xy=0$ for some homogeneous elements $x,y\in h(R)$ such that $x\not\in M'$ and $y\not\in M$.
 \end{enumerate}
\end{corollary}
\begin{proof}
$(1)\Rightarrow (2)$. Take $F=\{M\}$ and $F'=\{M'\}$ in the previous result.\par 
$(2)\Rightarrow (1)$. Let $P\in G\mathrm{Spec}(R)$. Assume $P\subseteq M\cap M'$ where $M\ne M'$ are $G$-graded maximal ideals of $R$. By hypothesis there are homogeneous elements $x,y\in h(R)$ such that $xy=0$, $x\not\in M$ and $y\not\in M'$. But $xy=0\in P$ implies that $x\in M\cap M'$ or $y\in M\cap M'$ a contradiction. 
\end{proof}
\begin{corollary}
If $R$ is a $G$-graded Gelfand ring, then $G\mathrm{Max}(R)$ is  Hausdorff. 
\end{corollary}
\begin{proof}
Follows immediately from the previous Corollary. 
\end{proof}
Recall that a topological space $X$ is normal or $T_4$ if, given any disjoint closed subsets $F$ and $F'$ of $X$, there are open  neighborhoods $U$ of $F$ and $V$ of $F'$ that are also disjoint. The next Theorem totally characterizes Gelfand graded rings by the normality of the homogeneous prime spectrum.
\begin{theorem}\label{Normal}
Let $R$ be a $G$-graded commutative ring. The following statements are equivalent.
\begin{enumerate}
\item $R$ is a $G$-graded Gelfand ring.
\item $G\mathrm{Spec}(R)$ is a normal space.
\end{enumerate} 
\end{theorem}
\begin{proof}
$(1)\Rightarrow (2)$. Let $I,J$ be a $G$-graded ideals of $R$ with $V_G(I)\cap V_G(J)=\emptyset$, that is $ I+J=R$. Consider $F:=V_G(I)\cap G\mathrm{Max}(R)$ and $F':=V_G(J)\cap G\mathrm{Max}(R)$. Then $F,F'$ are disjoint closed subsets of $G\mathrm{Max}(R)$. By Theorem \ref{Disjoint}, there are homogeneous elements $r,r'\in R$ such that $rr'=0$, $r\not\in \cup_{M\in F}M$ and $r'\not\in \cup_{M'\in F'}M'$. Set $U=D_G(r)$ and $V=D_G(r')$. Then $U\cap V=D_G(r)\cap D_G(r')=D_G(rr'=0)=\emptyset$. If $P\in V_G(I)$, then $P\subseteq M$ for some $M\in F$, so that $r\not\in P$. It follows that $V_G(I)\subseteq U $. By same argument $V_G(J)\subseteq V$. Thus $G\mathrm{Spec}(R)$ is a normal space.\par 
$(2)\Rightarrow (1)$. Let $M\ne M'$ be a $G$-graded maximal ideals of $R$. Since $\{M\}$ and $\{M'\}$ are disjoint closed subsets of $G\mathrm{Spec}(R)$, there are disjoint open subsets $U$ and $V$ of $G\mathrm{Spec}(R)$ such that $\{M\}\subseteq U$ and $\{M\}\subseteq V$, that is $M\in U$ and $M'\in V$. The collection $D_G(r)$ where $r\in h(R)$ forms a basis for that Zariski topology of $G\mathrm{Spec}(R)$, so that $M\in D_G(r)\subseteq U$ and $M'\subseteq D_G(r')\subseteq V$. Since $U$ and $V$ are disjoint, so are $D_G(r)$ and $D_G(r')$, that is $rr'$ is nilpotent, hence $r^nr'^n=0$ for some $n\in \Bbb N$. Now, observe that $a:=r^n\not\in M\cap M'$, $b:=r'^n\not\in M\cap M'$ and $ab=0$. By \ref{ExistsPrime}, there is no homogeneous prime ideal contained in $M\cap M'$. Therefore $R$ is a $G$-graded Gelfand ring. 
\end{proof}
For a topological space $X$ two subsets $A$ and $B$ are said to be separated by a continuous function if there exists a continuous function $f:X\to \Bbb R$ such that $f(x)=0$ for all $x\in A$ and $f(x)=1$ for all $x\in B$. The Urysohn's lemma is an interesting result in topology, it furnishes a characterization of normality in terms of continuous function. It states that a topological space is normal if and only if any two disjoint closed subsets can be separated by a continuous function. In the next result, we give an algebraic version of the Urysohn lemma. To do, we fix some notations and conventions. Let $R$ be a $G$-graded ring and $r\in R$. If $P$ is a homogeneous prime ideal of $R$ the canonical image of $r$ in $R_P$ is denoted $r(P)$. Note that $r(P)=0$ if and only if $ra=0$ for some $a\in h(R)\setminus P$. Two subsets $F$ and $F'$ of $G\mathrm{Spec}(R)$ are said to be separated by a regular function if there exists $r\in R$ such that $r(P)=0$ for all $P\in F$ and $r(P)=1$ for all $P\in F'$. Any such element $r$ is called a Urysohn regular function for $F$ and $F'$. Note that if $r$ is an Urysohn regular function for $F$ and $F'$ then so is $r_e$ where $r_e$ is the $e$-component of $r$. So, one can choose $r$ to be homogeneous element.  
\begin{theorem}
Let $R$ be a $G$-graded ring. The following statements are equivalent.
\begin{enumerate}
\item $R$ is a Gelfand graded ring.
\item Any two closed subsets of $G\mathrm{Spec}(R)$ can be separated by a regular function.
\end{enumerate}
\end{theorem}
\begin{proof}
$(1)\Rightarrow (2)$. Let $F$ and $F'$ be a disjoint closed subsets of $G\mathrm{Spec}(R)$. Set 
\begin{eqnarray*}
 I&=& \{ r\in R\ | \forall  P\in F, \exists a_P\in h(R)\setminus P \text{ such that } ra_P=0 \}\\ &=&\{r\in R\ | r(P)=0 \text{ for all } P\in F\ \}
 \end{eqnarray*}
 and $J= \{r\in R\ | r(P)=0 \text{ for all } P\in F'\ \}$. Then $I$ and $J$ are a graded ideals of $R$. We show that $I+J=R$. Assume to the converse that $I+J\ne R$. Then $I+J\subseteq M_0$ for some graded maximal ideal $M_0$. That is $I\subseteq M_0$ and $J\subseteq M_0$. \par 
 If $M_0\not\in F$, then $\{M_0\}$ and $F$ are a disjoint closed subsets of $G\mathrm{Spec}R$, as in the proof of  Theorem \ref{Normal}, there are $\alpha,\beta\in h(R)$ such that $\{M_0\}\subseteq D_G(\alpha)$ and $F\subseteq D_G(\beta)$ with $\alpha\beta=0$. We see that $\beta \in h(R)\setminus P$ for all $P\in F$ and $\alpha\beta =0$, that is $\alpha(P)=0$ for all $P\in F $. Thus $\alpha\in I\subseteq M_0$, hence $\alpha\in M_0$, which is not compatible with that fact that $M_0\in D_G(\alpha)$. It follows that $M_0\in F$. By the same way $M_0\in F'$. Which is a contradiction with the fact that $F$ and $F'$ are disjoint. It follows that $I+J=R$.\par 
 Now, let $r\in I $ and $r'\in J$ with $1=r+r'$. Since $I$ and $J$ are graded ideals of $R$, by taking the homogeneous components, one can assume that $r$ and $r'$ are homogeneous elements. Since $r\in I$, we have $r(P)=0$ for all $P\in F$. Also $r'(P)=0$ for all $P\in F'$, so $r(P)=1-r'(P)=1$ for all $P\in F'$.\\
 $(2)\Rightarrow (1)$. Let $M\ne M'\in G\mathrm{Max}(R)$. Then $\{M\}$ and $\{M'\}$ are disjoint closed subsets of $G\mathrm{Spec}(R)$. There exists  $r\in h(R)$ such that $r(M)=0$ and $r(M')=1$. We see that $rs=0$ for some $s\in h(R)\setminus M$ and $(r-1)t=0$ for some $t\in h(R)\setminus M'$. Now, observe that $st=0$. By the Corollary \ref{Max}, $R$ is a Gelfand graded ring.    
\end{proof}
The next result is an immediate consequence of the previous Theorem, that characterize Gelfand commutative ring in terms of regular functions.
\begin{corollary}
Let $R$ be a commutative ring. The following statements are equivalent.
 \begin{enumerate}
 \item $R$ is a Gelfand ring.
 \item Any two closed subsets of $\mathrm{Spec}(R)$ can be separated by a regular function.
\end{enumerate}  
\end{corollary} 
\begin{proof}
Follows from the fact that, if $R$ is a commutative ring then $R$ is naturally a $G$-graded commutative ring where $G=\{e\}$ is the trivial group. Moreover $G\mathrm{Spec}(R)=\mathrm{Spec}(R)$ and $h(R)=R$.
\end{proof}
Next we establish an algebraic criterion of Gelfand graded commutative rings.
\begin{theorem}\label{comax}
Let $R$ be a $G$-graded commutative ring. The following statements are equivalent.
\begin{enumerate}
\item $R$ is a Gelfand graded ring.
\item If $a\in R_e$, then there exists elements $b,c\in R_e$ such that $$(1-ba)(1-ca')=0$$ where $a'=1-a$. 
\end{enumerate}
\end{theorem}
\begin{proof}
$(1)\Rightarrow (2)$. Let $a\in R_e$ and $a'=1-a$. We see that $V_G(a)\cap V_G(a')=\emptyset$ that is $V_G(a)$ and $V_G(a')$ are disjoint closed subsets of $G\mathrm{Spec}(R)$. Since $R$ is Gelfand graded ring, by  Theorem \ref{Normal}, $G\mathrm{Spec}(R)$ is a normal space. There are disjoint open subsets $U$ and $V$ of $G\mathrm{Spec}(R)$ such that $V_G(a)\subseteq U$ and $V_G(a')\subseteq V$. Now, write $U=D_G(I)$ and $V=D_G(J)$ where $I$ and $J$ are homogeneous ideals of $R$. The condition $V_G(a)\subseteq D_G(I)$ implies that $V_G(a)\cap V_G(I)=\emptyset$, so that $(a)+I=R$. It follows that $1=\alpha a+i$ where $\alpha\in R$ and $i\in I$. Taking the $e$-components, we get $1=\alpha_0a+i_0$ where $\alpha_0\in R_e$ and $i_0\in I\cap R_e$. In the same way $1=\beta_0a'+j_0$ where $\beta_0\in R_e$ and $j_0\in J\cap R_e$. Since $U$ and $V$ are disjoint, every homogeneous element of $IJ$ is nilpotent. So that $i_0j_0$ is nilpotent, that is $(i_0j_0)^d=0$ for some $d\in \Bbb N$. This implies that $(1-\alpha_0a)^d(1-\beta_0a')^d=0$. By applying the binomial formula $(1-\alpha_0a)^d=1-d\alpha_0a+\ldots=1-ba$ where $b\in R_e$, also $(1-\beta_0a')=1-ca'$ where $c\in R_e$.Therefore $(1-ba)(1-ca')=0$. \par 
$(2)\Rightarrow (1)$. Let $M\ne M'$ be a $G$-graded maximal ideals of $R$. Since $M+M'=R$, $1=a+a'$ where $a\in M$ and $a'\in M'$. By taking the $e$-component we may assume that $a$ and $a'$ are in $R_e$. By hypothesis there are elements $b,c\in R_e$ such that $(1-ba)(1-ca')=0$. Now, observe that $1-ba\not\in M$ and $1-ca'\not\in M'$. By Corollary \ref{Max}, $R$ is a Gelfand $G$-graded ring. 
\end{proof}
The following Corollary is an immediate consequence of the previous Theorem, that characterize Gelfand graded ring in terme of its  subring $R_e$.  
\begin{corollary}
Let $R$ be a $G$-graded commutative ring. Then $R$ is a Gelfand graded ring if and only if $R_e$ is a Gelfand ring.
\end{corollary}
\begin{proof}
This follows immediately from the previous characterization.
\end{proof}
\begin{example}
\begin{enumerate}
\item Let $R$ be a commutative ring and consider $R[X]$ as a $\Bbb Z$-graded ring. Then $R[X]$ is a Gelfand $\Bbb Z$-graded ring if and only if $R$ is a Gelgand ring.
\item Let $A$ be a commutative ring and $M$ be an $A$-module. Let $R:=A\propto M$ be the trivial ring extension of $A$ by $M$. Then $R$ is a $\Bbb Z_2$-graded ring with $R_0=A$ and $R_1=M$, moreover  the homogeneous prime spectrum of $R$ and the prime spectrum of $R$ are the same, see \cite{Aqa}. It follows that $R$ is a Gelfand ring if and only if $A$ is a Gelfand ring. 
\end{enumerate}
\end{example}
\section{ pm$^+$ graded  rings}
In this section we introduce and study a class of graded rings that behave as a graded rings with a Gelfand strong property, called pm$^+$ graded ring.
\begin{definition}
Let $R$ be a $G$-graded commutative ring. Then $R$ is called a pm$^+$ graded ring if for each $P\in G\mathrm{Spec}(R)$, the homogeneous prime ideals of $R$ containing $P$ form a chain (i.e linearly ordered). 
\end{definition}
\begin{example}
\begin{enumerate}
\item Valuation graded ring is a pm$^+$ graded ring.
\item If $h\dim R=0$, then $R$  is a pm$^+$ graded ring.
\end{enumerate}
\end{example}
\begin{proposition}\label{QotLoc}
Let $R$ be a pm$^+$ graded commutative ring.
\begin{enumerate}
\item $R$ is a Gelfand graded ring.
\item If $S\subseteq h(R)$ is a multiplicatively closed subset, then $R_S$ is a pm$^+$ graded ring.
\item If $I$ is a graded ideal of $R$, then $R/I$ is a pm$^+$ graded ring.
\end{enumerate}
\end{proposition}
\begin{proof}
\begin{enumerate}
\item Immediate.
\item This follows from the fact that the set of homogeneous  prime ideals of $R_S$ is in bijection (order preserving) with the set of homogeneous  prime ideals of $R$ disjoint with $S$. 
\item Immediate.
\end{enumerate}
\end{proof}
The following is a simple algebraic characterization of pm$^+$ graded commutative ring.
\begin{proposition}\label{pmbyelement}
Let $R$ be a $G$-graded commutative ring. The following statements are equivalent.
\begin{enumerate}
\item $R$ is a pm$^+$ graded ring.
\item If $P,Q\in G\mathrm{Spec}(R)$ are not comparable, then there are $r,r'\in h(R)$ such that $rr'=0$ and $r\not\in P$, $r'\not\in Q$. 
\end{enumerate}
\end{proposition}
\begin{proof}
$(1)\Rightarrow (2)$. If $P,Q\in G\mathrm{Spec}(R)$ are not comparable, then there is no homogeneous prime ideal contained in $P\cap Q$ since $R$ is a pm$^+$ graded ring. By \ref{ExistsPrime}, there are elements $r,r'\in h(R)$ such that $rr'=0$ and $r,r'\not\in P\cap Q$. So, necessarily one is not in  $P$ and the other is not in $Q$. \par 
$(2)\Rightarrow (1)$. Let $P_0\in G\mathrm{Spec}(R)$ and $P,Q\in G\mathrm{Spec}(R)$ containing $P_0$. If $r,r'\in h(R)$ with $rr'=0$, then $rr'\in P_0\subseteq P\cap Q$, so that $r\in P\cap Q$ or $r'\in P\cap Q$. It follows from the hypothesis that $P$ and $Q$ are comparable. 
\end{proof}
\begin{proposition}
Let $R$ be a $G$-graded commutative ring. The following statements are equivalent.
\begin{enumerate}
\item $R$ is a pm$^+$ graded ring.
\item For every minimal homogeneous prime ideal $m$ of $R$, the homogeneous prime ideals of $R/m$ are linearly ordered.
\end{enumerate}
\end{proposition}
\begin{proof}
$(1)\Rightarrow (2)$. No comment.\par 
$(2)\Rightarrow (1)$. Let $P\in G\mathrm{Spec}(R)$, then $P$ contain a minimal homogeneous prime ideal $m$. Since the set of homogeneous prime ideals of $R$ containing $P$ is a closed subset of $R/m$, it is linearly ordered. Thus $R$ is pm$^+$ graded ring.  
\end{proof}
For $P\in G\mathrm{Spec}(R)$, denote $$C(P)=\{Q \in G\mathrm{Spec}(R)\ | P\subseteq Q \text{ or } Q\subseteq P \}$$ Then $C(P)$ is the set of all homogeneous prime ideals of $R$ that are comparable with $P$, note that, if $P$ is a graded maximal ideal of $R$, then $C(P)=\{ Q\in \mathrm{Spec}(R)\ | Q\subseteq P\}$. The following result is a topological characterization of pm$^+$ graded ring in terms of the subsets $C(P)$.
\begin{theorem}
Let $R$ be a $G$-graded commutative ring. The following statements are equivalent.
\begin{enumerate}
\item $R$ is a pm$^+$ graded ring.
\item For every $P\in G\mathrm{Spec}(R)$, $C(P)$ is a closed subset of $G\mathrm{Spec}(R)$.
\end{enumerate}
\end{theorem}
\begin{proof}
$(1)\Rightarrow (2)$. Pick $P\in G\mathrm{Spec}(R)$ and set $I=\cap_{Q\subseteq P} Q$ where $Q$ runs through homogeneous prime ideals of $R$ contained in $P$. We show that $C(P)=V_G(I)$. It is easy to see $C(P)\subseteq V(I)$. Now, let $Q'\in V_G(I)$. Let $r,r'\in h(R)$ such that $rr'=0$. We will show that $r$ or $r'$ is in $P\cap Q'$. So, we may assume that $r\not\in P\cap Q'$. Hence $r\not\in P$ or $r\not\in Q'$.\\
\textbf{First case ;} $r\not\in Q'$. In this case $r'\in Q'$. Observe that $r\not\in I$ since $I\subseteq Q'$. It follows that $r\not\in Q$ for some $Q\subseteq P$, thus $r'\in Q$ since $rr'=0\in Q$. Therefore  $r'\in P$. Thus $r'\in P\cap Q'$.\\
\textbf{Second case ;} $r \not\in P$. In this case $r'\in P$. We see that $r\not\in Q$ for all $Q\subseteq P$, so that $r'\in Q$ for all $Q\subseteq P$ since $rr'=0\in Q$. It follows that $r'\in I=\cap_{Q\subseteq P}Q$, therefore $r'\in Q'$. Hence $r'\in P\cap Q'$.\\
Now, by Lemma \ref{ExistsPrime}, there exists a homogeneous prime ideal $P'$ of $R$ such that $P'\subseteq P\cap Q'$, that is $P$ and $Q'$ are homogeneous prime ideals of $R$ containing $P'$. Since $R$ is a pm$^+$ graded ring, $P$ and $Q'$ are comparable, that is $Q'\in C(P)$.\par 
$(2)\Rightarrow (1)$. Let $P\in G\mathrm{Spec}(R)$ and $Q,Q'\in G\mathrm{Spec}(R)$ be homogeneous prime ideals of $R$ containing $P$. We have $P\in C(Q)$  and $C(Q)$ is a closed subset of $G\mathrm{Spec}(R)$, so $V_G(P)\subseteq C(Q)$, therefore $Q'\in C(Q)$ since $Q'\in V_G(P)$, this means that $Q$ and $Q'$ are comparable.   
\end{proof}
Next, we introduce  the following conditions in order to realize a relation between pm$^+$ graded ring and Gelfand graded ring:\\
 (Av) If $I,J,J'$ are graded ideals of $R$ such that $h(R)\cap I\subseteq J\cup J'$, then $I\subseteq J$ or $I\subseteq J'$.\\
 (Cm) If $J$ and $J$ are  not comparable graded ideals of $R$, then there exists $g\in G$ and $r,r'\in R_g$ such that $r\in J-J'$ and $r'\in J'-J$.\\
 (PCm) If $Q$ and $Q'$ are  not comparable homogeneous prime ideals of $R$, then there exists $g\in G$ and $r,r'\in R_g$ such that $r\in Q-Q'$ and $r'\in Q'-Q$.
 \begin{proposition}
 Let $R$ be $G$-graded commutative ring. Then the two conditions (Av) and (Cm) are equivalent.
 \end{proposition}
 \begin{proof}
 Assume that $R$ satisfies the condition (Av). Let $J,J'$ be two graded ideals of $R$ which are not comparable. Assume to the contrary that, for every $g\in G$, $R_g\cap J\subseteq R_g\cap J'$ or $ R_g\cap J'\subseteq R_g\cap J$. Let $x\in R_g\cap (J+J')$ where $g\in G$, then $x=a+b$ where $a\in R_g\cap J$ and $b\in R_g\cap J'$, so $a+b\in R_g\cap J$ or $a+b\in R_g\cap J'$ since $R_g\cap J\subseteq R_g\cap J'$ or $ R_g\cap J'\subseteq R_g\cap J$. It follows that $x\in J\cup J'$. So $h(R)\cap (J+J')\subseteq J\cup J'$. The condition (Av) implies that $J+J'\subseteq J$ or $J+J'\subseteq J$, so that $ J'\subseteq J$ or $J\subseteq J'$, a contradiction.\\
 Assume that $R$ satisfies the condition (Cm). Let $I,J,J'$ be a graded ideals of $R$ such that $h(R)\cap I\subseteq J\cup J' $. Set $K=I\cap J$ and $K'=I\cap J'$, then $K$ and $K'$ are graded ideals of $R$. If $K$ and $K'$ are not comparable, then by the condition (Cm), there are $g\in G$, $r,r'\in R_g$ such that $r\in K-K'$ and $r'\in K'-K$. We see that $r,r'\in R_g\cap I$, so $r+r'\in h(R)\cap I$. Therefore $r+r'\in J$ or $r+r'\in J'$. This implies that $r\in J'$ or $r'\in J$ since $r\in J$ and $r'\in J'$, a contradiction. Thus $K$ and $K'$ are comparable. Without lost of generality, we may assume that $K\subseteq K'$, that is $I\cap J\subseteq I\cap J'$. So that 
 $h(R)\cap I\subseteq (J\cap I)\cup (J'\cup I)\subseteq J'$, hence $I\subseteq J'$. 
 \end{proof}
 \begin{remark}\label{Torsion} If $G$ is a torsion group (e.g. finite group) and $R$ is a $G$-graded commutative ring, then $R$ satisfies the condition (PCm). Indeed, let $Q$, $Q'$ be a not comparable homogeneous prime ideals of $R$. Then there are $a,b\in h(R)$ with $a\in Q-Q'$ and $b\in Q'-Q$. We have $a\in R_g$ and $b\in R_h$ for some $g,h\in G$. Since $G$ is a torsion group, $g^n=e$ and $h^m=e$ for some $n,m\in \Bbb N$. Observe that $r:=a^n$ and $r':=b^m$ are homogeneous elements of $R$ of same degree ($r,r'\in R_e$) such that $r\in Q-Q' $ and $r'\in Q'-Q$. 
 \end{remark}
\begin{theorem}
Let $R$ be a $G$-graded commutative ring. Consider the following statements.
\begin{enumerate}
\item $R$ is a pm$^+$ graded ring.
\item If $S\subseteq h(R)$ is a multiplicatively closed subset of $R$, then $R_S$ is a Gelfand graded ring.
\item If $r\in h(R) $ is a non nilpotent element of $R$, then $R_r$ is a Gelfand graded ring. 
\item If $r\in h(R) $ is a non nilpotent element of $R$, then $D_G(r)$ is a normal space. 
\end{enumerate}
Then $1)\Rightarrow 2) \Rightarrow 3)\Rightarrow 4)$. Moreover, if  $R$ satisfies the condition  (PCm) (in particular if $R$ satisfies that condition (Av)), then $4)\Rightarrow 1)$.
\end{theorem}
\begin{proof}
$(1)\Rightarrow (2)$.  From \ref{QotLoc}. $(2)\Rightarrow (3)$ Take $S=\{r^n\ / n\in \Bbb N\}$. $(3)\Leftrightarrow (4)$ By Theorem  \ref{Normal} and the fact that $D_G(r)$ and $G\mathrm{Spec}(R_r)$ are homeomorphic.\\
Assume that $R$ satisfies the conditions (PCm) and $(3)$. Let $Q$ and $Q'$ be a not comparable homogeneous prime ideals of $R$. The condition (PCm) implies that there are $g\in G$, $r,r'\in R_g$ such that $r\in Q-Q'$ and $r'\in Q'-Q$. Set $t=r+r'$, then $t$ is a non nilpotent homogeneous element of $R$ since $t\not\in Q$. Note that $r/t$ and $r'/t$ are homogeneous elements of $R_t$ of degree $e$.  Since $r/t+r'/t=1$ and $R_t$ is a Gelfant graded ring, by Theorem \ref{comax}, there are $a/t^n, b/t^m\in (R_t)_e$  such that $(1-(r/t) (a/t^m))(1-(r'/t) (b/t^m))=0$. Multiplying by a suitable power of $t$, it follows   that $(t^k-ra')(t^k-r'b')=0$ for some $k\in \Bbb N$ where  $a'=at^{k-n-1}$ and $b'=bt^{k-m-1}$. Take $u=t^k-ra'$ and $v=t^k-r'b'$. We have 
$$\deg(ra')=\deg(r)\deg(a)\deg(t^{k-n-1})=gg^ng^{k-n-1}=g^k=\deg(t^k)$$ It follows that $u\in R_{g^k}$. We see also that $v\in R_{g^k}$. Now, observe that $uv=0$ and $u,v\in R_{g^k}$. Moreover, $u\not\in Q$, in fact if $u\in Q$, then $t^k-ra'\in Q $, so that $t^k\in Q$, that is $t\in Q$, a contradiction. By a same argument $v\not\in Q'$. It follows from Proposition \ref{pmbyelement} that $R$ is a pm$^+$ graded ring.   
\end{proof}
\begin{corollary}
Let $G$ be a torsion group and $G$ be a $G$-graded commutative ring. The following statements are equivalent.
\begin{enumerate}
\item $R$ is a pm$^+$ graded ring.
\item If $S\subseteq h(R)$ is a multiplicatively closed subset, then $R_S$ is a Gelfand graded ring.
\item If $t\in h(R) $ is a non nilpotent element of $R$, then $R_t$ is a Gelfand graded ring. 
\item If $t\in h(R) $ is a non nilpotent element of $R$, then $D_G(t)$ is a normal space. 
\end{enumerate}
\end{corollary}
\begin{proof}
This is an immediate consequence of the previous Theorem and the Remark \ref{Torsion}. 
\end{proof}

\end{document}